\documentclass{amsart}
\usepackage{amsmath,amssymb,latexsym,verbatim,color, hyperref}

\newtheorem{theorem}{Theorem}[section]

\newtheorem{proposition}{Proposition}[section]

\theoremstyle{definition}

\begin{document}
\title{On a Reformulation of the Commutator Subgroup}

\author{ Paul. A. Cummings} 
\author{Brian Ortega}

\date{September, 2011}

\begin{abstract} 

For semigroup $S$, a commutative congruence $\sigma_{orient}$ on $S$ and a subsemigroup Orientable($S$) of $S$ were both introduced in \cite{CCT}. Here we demonstrate that when the semigroup is in fact a group $G$, then Orientable($G$) is its commutator subgroup $[G,G]$ and $G / \sigma_{orient}$ is the abelian quotient group $G / [G,G]$.
\end{abstract}

\maketitle

\hbox{\hskip 240pt Version of \the\month-\the\day-\the\year}

\setcounter{section}{-1}
\section{Introduction} \label{s_intro}

We will come to these reformulations through the theory of semigroups. Recall a semigroup $S$ is a set equipped with a closed associative binary operation. A group $G$ can be thought of as a semigroup. However in general, unlike groups, semigroups need not contain an identity element or inverse elements, although there are semigroups with an identity element as well as semigroups with inverse elements. See \cite{H} or \cite {P} for instance. Herein, though, when we speak of a semigroup we will only be interested in it as a set closed under some specified associative binary operation unless otherwise stated. 

The elements in the subsemigroup Orientable($S$) are characterized as those elements of $S$ that are solutions to an \emph{orientable} equation, a certain type of simple equation in one variable over $S$. Two elements $u$ and $v$ in $S$ are $\sigma_{orient}$-congruent to one another if they are solutions to a two-variable orientable equation, a natural extension of the one variable variety. We note and the reader can verify this after we review orientable equations in Section \ref{orientable equations} that if $S$ did contain an identity element then the elements belonging to Orientable ($S$) could also be characterized as those elements that are $\sigma_{orient}$-congruent to $S$'s identity element. 

Now in a certain sense, as somewhat indicated in the above paragraph, the subsemigroup Orientable($S$) determines the congruence $\sigma_{orient}$ on $S$ and, as demonstrated in \cite{CCT}, the quotient semigroup $S / \sigma_{orient}$ is a commutative semigroup and 
Orientable($S)$ serves as its identity element, in the case when Orientable($S$) is non-empty. Already, at least, this ``looks like'' and ``acts like'' how the commutator subgroup $[G,G]$ determines a congruence on $G$ that gives rise to a quotient group $G / [G,G]$ that is a commutative group, the abelianization of $G$, and where the commutator subgroup $[G,G]$ serves as its identity element.

To further strengthen these resemblances it was also shown in \cite{CCT} - using diagrams embedded on orientable surfaces and hence the origin of the modifier orientable for orientable equations - that when semigroup $S$ and group $G$ are both defined by the same semigroup presentation $P = \langle \hspace{.1cm} X \hspace{.1cm} \vert \hspace{.1cm} R  \hspace{.1cm} \rangle$, then $w$ belongs to Orientable($S$) if and ony if $w$ belongs to $[G,G]$ where $w$ is a positive word on the alphabet $X$ and the natural map from $S$ into $G$ factors through $\sigma_{orient}$ as an embedding into $G / [G, G]$. The aim of this article is to show that when the semigroup $S$ is itself a group $G$ then these resemblances are in fact identically the same.

In Section \ref{orientable equations} we review orientable equations and reference some propositions concerning same found in \cite{CCT}. We will also review and reference material concerning the commutator subgroup and the abelianization of a group. In Section \ref{main results} we demonstrate our main results, namely, $ G/\sigma_{orient} = G/[G,G]$ and Orientable($G$) $=$ $[G,G]$ for any group $G$. Our proofs will be carried out simply by mathematical induction, at this stage there will be no need for the geometry of diagrams embedded on orientable surfaces.

\section{Orientable Equations}  \label{orientable equations}

The following definitions and results - with a slight change in notation, namely, we now use only lower case letters representing elements in $S$ - are taken directly from \cite{CCT}. Our main results in Section \ref{main results} will simply follow from mathematical induction and from the definitions of equations $(1)$ and $(2)$ below. We include the following propositions merely as context for $(1)$ and $(2)$. Let $S$ be a semigroup. We pick a symbol $t\notin S$ to be a variable. A {\it  simple equation over $S$ in the variable $t$} is an expression of the form 
\bigskip 
\begin{center}  $a=btc$ \qquad$(1)$\end{center} 
\bigskip

\noindent where $a,b,c\in S$, $a$ is non-empty, and least one of the elements $b$ and $c$ is non-empty. An element $w \in S$ satisfies this equation if upon substitution $a=bwc$ in $S$. If there are non-empty elements $x_1, ..., x_n$ (not necessarily distinct) in $S$ and factorizations $a=\prod_{i=1}^{n_1}a_i$, $b=\prod_{j=1}^{n_2}b_j$ and $c=\prod_{k=1}^{n_3}c_k$ where $n_1 + n_2 + n_3 = 2n$, $n_1 = n_2 + n_3$ and each $x_\nu$ occurs once among $a_1, ..., a_{n_1}$ and once among $b_1, ..., b_{n_2}, c_1, ..., c_{n_3}$ then we call the simple equation an {\it orientable equation}. In other words a simple equation is orientable if the factors $x_\nu$ can be somehow paired up across the equality sign of the simple equation.

Elements of $S$ that satisfy orientable equations are called {\it orientable} elements of $S$ and by {\it Orientable($S$)} we mean the set $\{  w \hspace{.1cm} | \hspace{.1cm}  w \hspace{.1cm}  $is an orientable element of$ \hspace{.1cm} S \}$.
\bigskip

The following appears in Proposition 1.1 in \cite{CCT}. In fact, there it is also claimed that Orientable($S$) is a {\it unitary} subsemigroup of $S$ although we will have no need of this concept and fact.

\begin{proposition} \label{orientable(S)} Let S be a semigroup then Orientable(S) is a subsemigroup of S.
\end{proposition}

We now introduce two-variable orientable equations, modeled after the single variable orientable equations.  
\bigskip

 Let $S$ be a semigroup. We pick two symbols $t_{1}$ and $t_{2}$ both $\notin S$ to be variables. A {\it two-variable simple equation over $S$} in the variables $t_{1}$ and $t_{2}$ is an expression of the form \bigskip \begin{center}  $at_{1}b=ct_{2}d$ \qquad$(2)$\end{center} \bigskip

\noindent where $a,b,c, $and$ d \in S$, some or possibly all of which may be empty. An ordered pair $(u,v)$ $\in S \times S$ satisfies this equation if upon substitution $aub=cvd$ in $S$. Analogous to the one-variable equations of section 1, if there are non-empty elements $x_1, ..., x_n$ (not necessarily distinct) in $S$ and factorizations $a=\prod_{i=1}^{n_1}a_i$, $b=\prod_{j=1}^{n_2}b_j$, $c=\prod_{k=1}^{n_3}c_k$, and $d=\prod_{l=1}^{n_4}d_l$  where $n_1 + n_2+n_3 + n_4=2n$, $n_1 + n_2 = n_3 + n_4$ and each $x_\nu$ occurs once among $a_1, ..., a_{n_1}, b_1, ..., b_{n_2}$ and once among $c_1, ..., c_{n_3}, d_1, ..., d_{n_4}$ then we call such two-variable simple equation a {\it two-variable orientable equation over S}. 

\bigskip

 Let $u$ and $v$ be two elements in $S$. We say $u$ is {\it orientably-equivalent} to $v$ in $S$ and write $ u \sigma_{orient} w$ if the ordered pair $(u,v)$ satisfies a two-variable orientable equation over $S$. 

\bigskip

The following proposition follows from Claims 4.1, 4.2, 4.3, and 4.4 in \cite{CCT}. 

\begin{proposition} \label{congruence} Let $S$ be a semigroup then $\sigma_{orient}$ is a congruence on $S$.
\end{proposition}

\bigskip

The following follows from Proposition 4.1 in \cite{CCT} where it is also shown that the quotient semigroup $S/\sigma_{orient}$ is {\it cancellative} although we will nave no need for this concept or fact.

\begin{proposition} \label{commutative} Let $S$ be a semigroup then the quotient semigroup $S/\sigma_{orient}$ is a commutative semigroup.
\end{proposition}

The following follows from Propostion 4.3 in \cite{CCT}

\begin{proposition} \label{identity} Let $S$ be a semigroup. Then $Orientable(S)$ is the identity element for $S/\sigma_{orient}$, when Orientable($S$) is non-empty. 
\end{proposition}

\section{Main Results} \label{main results}   

Recall that for group $G$ a {\it commutator} in $G$ is an element of the form $xyx^{-1}y^{-1}$ where $x$ and $y$ are elements of $G$ and $x^{-1}$ and $y_{-1}$ are their inverses, respectively.  The commutator subgroup $[G,G]$ of $G$ is the smallest subgroup of $G$ containing all of $G$'s commutators. It is well known that $[G,G]$ is a normal subgroup and that $[G,G]$ determines a congruence $\sim$ on $G$ where for elements $g$ and $h$ of $G$ we have $g \sim h$ when $gh^{-1} \in [G,G]$. It is also well known that $\sim$ determines the abelian (commutative) quotient group $G/[G,G]$ - the abelianization of $G$ - where $[G,G]$ serves as its identity element. All of these claims can be found in most undergraduate abstract algebra textbooks, we cite two \cite{Hu, L}. For element $g \in G$ we will use the notation $[g]$ to denote its congruence class in the $G/[G,G]$, i.e. $[g] = \{ h \vert h \in G \hspace{.1cm} $and$ \hspace{.1cm} h \sim g \}$.

\begin{theorem} \label{Orientable(G)} If $G$ is a group then Orientable($G$) $=$ $[G,G]$, the commutator subgroup of $G$. 
\end{theorem}

\begin{proof} $\Rightarrow$ Let $g$ $\in$ $G$ that satisfies some orientable equation $a=btc$. Since $a=bgc$ in $G$ then clearly $[a] = [b][g][c]$ in $G / [G,G]$. As $G / [G,G]$ is abelian, $[b^{-1}][a][c^{-1}] = [g]$. Since $a=btc$ is an orientable equation and $G/[G,G]$ is abelian, clearly $[b^{-1}][a][c^{-1}] = [1]$. Hence $[g] = [1]$ and so $g$ $\in$ $[G,G]$.  

$\Leftarrow$ Let $g$ $\in$ $[G,G]$. So there must exist elements $x_{i}, y_{i} \in G$ for $1 \leq i \leq n$ such that $g = \prod_{i=1}^{n}x_{i}y_{i}x_{i}^{-1}y_{i}^{-1}$. We induct on $n$. For $n=1$ we have $x_{1}y_{1}x_{1}^{-1}y_{1}^{-1} = g$.  Hence $x_{1}y_{1} = gy_{1}x_{1}$ and therefore $g$ satisfies the orientable equation $x_{1}y_{1} = ty_{1}x_{1}$. So next assume $g = \prod_{i=1}^{n}x_{i}y_{i}x_{i}^{-1}y_{i}^{-1}$. Hence $gy_{n}x_{n}y_{n}^{-1}x_{n}^{-1} = \prod_{i=1}^{n-1}x_{i}y_{i}x_{i}^{-1}y_{i}^{-1}$ and therefore by induction $gy_{n}x_{n}y_{n}^{-1}x_{n}^{-1}$ satisfies some orientable equation $a=btc$, i.e., $a=bgy_{n}x_{n}y_{n}^{-1}x_{n}^{-1}c$. Since $x_{n}x_{n}^{-1}y_{n}y_{n}^{-1} = 1$ then $ax_{n}x_{n}^{-1}y_{n}y_{n}^{-1} = bgy_{n}x_{n}y_{n}^{-1}x_{n}^{-1}c$. Now lastly observe that  $ax_{n}x_{n}^{-1}y_{n}y_{n}^{-1} = bty_{n}x_{n}y_{n}^{-1}x_{n}^{-1}c$ is also an orientable equation that $g$ satisfies. 

\end{proof}

\begin{theorem} If $G$ is a group then $G / \sigma_{orient} = G / [G,G]$, the abelianzation of $G$.
\end{theorem}

\begin{proof} Let $g$ be some element of group $G$ and consider the congruency classes $[g]_{\sigma}$ and $[g]$ of $G / \sigma_{orient}$ and $G / [G,G]$, respectively. We will show $[g]_{\sigma} = [g]$. So consider $h \in [g]_{\sigma}$. Hence $h$ and $g$ satisfy a two-variable orientable equation as given in $(2)$. Therefore for some elements $a$, $b$, $c$, and $d$ in $G$ we have $agb = chd$ in $G$. Hence [agb] = [chd] in $G / [G,G]$. Hence [ab][g] = [cd][h] in $G / [G,G]$ as $G / [G,G]$ is abelian. By the definition of a two-variable orientable equation it easy to see that we must also have [ab]= [cd] in $G / [G,G]$. Hence $[g] = [h]$ and therefore $[g]_{\sigma} \subseteq [g]$. Lastly, consider $h \in [g]$. Hence $gh^{-1} \in [G,G]$. Therefore, by Theorem \ref{Orientable(G)}, $gh^{-1} \in Orientable(G)$. Hence $gh^{-1}$ satisfies an orientable equation as given in $(1)$. Therefore there exists elements $a$, $b$, and $c$ in $G$ such that $a = bgh^{-1}c$ in $G$. It is easy to to check that $at_{1}h^{-1} = bt_{2}h^{-1}c$ is a two-variable orientable equation over $G$ and that $g$ and $h$ satisfy it upon substituting $h$ into $t_1$ and $g$ into $t_2$.  Therefore $h \in [g]_{\sigma}$ and $[g] \subseteq [g]_{\sigma}$. 
\end{proof}

\section{Acknowledgments} \label{acknowledge}

The authors would like to express their deep gratitude to Maritza Martinez, Director of the Educational Opportunities Program at the University at Albany, State University of New York, for her support and encouragement throughout this undergraduate research project. 

We would also like to thank Anahi Bolanos, Grace Coste, Stephine Rodriguez, and Yasser Teouri for their valued participation in a weekly math seminar in part devoted to the second named author's presentation of some of the materials herein.

\end{document}